\documentclass{amsproc}

\usepackage{amssymb}

\usepackage{graphicx}

\usepackage{enumerate}
\usepackage{}

\newtheorem{theorem}{Theorem}[section]
\newtheorem{lemma}[theorem]{Lemma}
\newtheorem{conjecture}[theorem]{Conjecture}

\theoremstyle{definition}
\newtheorem{definition}[theorem]{Definition}

\theoremstyle{remark}
\newtheorem{remark}[theorem]{Remark}

\numberwithin{equation}{section}

\newcommand{\Hh}{\mathbb{H}}
\newcommand{\D}{\mathbb{D}}
\newcommand{\B}{\mathbb{B}}
\newcommand{\C}{\mathbb{C}}

\newcommand{\rea}{\operatorname{Re}}

\newcommand{\Arg}{\operatorname{Arg}}

\begin{document}

\title{On convergence to the Denjoy-Wolff point in the parabolic case}


\author{Olena Ostapyuk}
\curraddr{Department of Mathematics,
University of Northern Iowa,
Cedar Falls, Iowa 50614-0506}
\email{ostapyuk@math.uni.edu}



\date{12/24/2011}

\begin{abstract}
Based on dynamical behavior, all self-maps of the unit disk in the complex plane can be classified as elliptic, hyperbolic or parabolic. The parabolic case is the most complicated one and branches into two subcases - zero-step and non-zero-step cases. In several dimensions, zero-step and non-zero step cases can be defined for sequences of forward iterates, but it is not known yet if the classification can be extended to parabolic maps of the ball. However, some geometric properties of the forward iterates can be generalized to higher-dimensional case. 
\end{abstract}

\maketitle


\section{Introduction}

Consider an analytic self-map of the open unit disk $\D$ in the complex plane, i.e. the function $f$ such that $f(\D)\subseteq\D$. The classical Schwarz's lemma says, that if $f(0)=0$, then 
\[
|f(z)|\leq|z|\hspace{.5in}{and}\hspace{.5in}|f^{\prime}(0)|\leq 1,
\]
and if equality holds for a point $z\neq 0$, then $f(z)=cz$ with $|c|=1$ (rotation around the center). In other words, unless $f$ is a rotation, the (Euclidean) distance between images of $z$ and $0$ is smaller then the distance between $z$ and $0$.

A similar statement about any two points in the unit disk holds, if we replace Euclidean distance by pseudo-hyperbolic distance:
\begin{theorem}[Point-invariant form of Schwarz's lemma]\
For any analytic self-map $f$ of the unit disk and any $z,w\in\D$,
\[
\left|\frac{f(z)-f(w)}{1-f(z)\overline{f(w)}}\right|\leq\left|\frac{z-w}{1-z\overline{w}}\right|,
\]
and equality holds for some distinct pair iff $f$ is an automorphism of $\D$;
i.e. $f$ is contraction in the pseudo-hyperbolic metric $d(z,w)=\left|\frac{z-w}{1-z\overline{w}}\right|$:
\begin{align}
d(f(z),f(w))\leq d(z,w).\label{SL}
\end{align}
\end{theorem}

Denote $f_n=f^{\circ n}$ and consider the sequence of forward iterates of $f$, i.e. $z_n=f_n(z_0)$. By Schwarz's lemma, the sequence $d(z_n,z_{n+1})$ is non-increasing; moreover, as the following theorem states, except for the case of an elliptic automorphism, all forward iteration sequences must converge to the same point in the closed disk:

\begin{theorem}[Denjoy-Wolff \cite{Wolff1}, \cite{Wolff2}, \cite{Wolff3} and \cite{Denj}] 
If $f$ is not an elliptic automorphism, then there exists a unique point $p\in\overline{\D}$ (called the Denjoy-Wolff point of $f$) such that the sequence of iterates $\{f_n\}$ converges to $p$ uniformly on compact subsets of $\D$.
\end{theorem}

Based on their behavior near the Denjoy-Wolff point $p$, we can classify self-maps of the disk as follows:
\begin{enumerate}
	\item {If the Denjoy-Wolff point $p$ is inside of the unit disk, then $f$ is called {\bf elliptic}. The point $p$ is a fixed point of $f$ (i.e. $f(p)=p$) and $|f^{\prime}(p)|\leq 1$. When $|f^{\prime}(p)|=1$, $f$ is an elliptic automorphism (up to change of variables, rotation around the center).}
	\item{If the Denjoy-Wolff point $p$ is on the boundary of the unit disk and $f^{\prime}(p)<1$ (in the sense of non-tangential limits), then $f$ is called {\bf hyperbolic}. $p$ is again a fixed point of $f$, now in the sense of non-tangential limits. Forward iterates tend to the Denjoy-Wolf point along non-tangential directions.}
	\item{If the Denjoy-Wolff point $p$ is on the boundary of the unit disk and $f^{\prime}(p)=1$ (in the sense of non-tangential limits), then $f$ is called {\bf parabolic}. Similarly to the hyperbolic case, $p$ is a fixed point of $f$; but now forward iterates may converge tangentially as well as non-tangentially to the boundary.}
\end{enumerate}

In this paper we will discuss the behavior of forward iterates in the parabolic case in the unit disk $\D$ and in the ball $\B^N$.

In some cases, it will be convenient to use half-plane $\Hh=\left\{z\in\C \left|\right.\rea z>0\right\}$ or Siegel half-plane $\Hh^N=\left\{(z,w)\in \C\times\C^{N-1} : \rea z>\|w\|^2\right\}$, which are biholomorphically equivalent to the unit disk $\D$ and to the unit ball $\B^N$, respectively. Without loss of generality, we can always assume that in these models, the Denjoy-Wolff point is $\infty\in\overline{\Hh}$, or  $\infty\in\overline{\Hh^N}$.

\section{One-dimensional (unit disk) case}
Note that by (\ref{SL}), the pseudo-hyperbolic distance between two consecutive forward iterates $d(z_n,z_{n+1})$ is non-increasing and thus has a limit $d_{\infty}$. Whether this limit is positive or zero defines the behavior of the sequence and the function.

\begin{definition}\label{def:znz}
We will call a sequence $\{z_n\}$ a zero-step (respectively, non-zero-step) sequence, if $d_{\infty}=\lim d(z_n,z_{n+1})=0$ (respectively, $d_{\infty}>0$).
\end{definition}

In the one-dimensional case, as a consequence of the theorem of Pommerenke (Theorem \ref{thm:Pom1} below), zero-step and non-zero-step properties of a sequence of forward iterates do not depend on the choice of the starting point but depend on the function only, so we can call functions parabolic zero-step and parabolic non-zero step, respectively (see \cite{PPC2}). It is still not known if the same is true in several variables.

\begin{remark}
Here we follow the terminology introduced in \cite{PPC2}; Pommerenke in \cite{Po} and \cite{BPo} used the term "parabolic" for the parabolic non-zero-step case and "identity" for the parabolic zero-step case. 
\end{remark}

More about the parabolic non-zero-step and zero-step cases in one dimension, including backward iteration and examples, can be found in \cite{PPC2}.  

The crucial difference between parabolic non-zero-step and zero-step functions in the unit disk is that the former are conjugated to a (vertical) translation in the half-plane:
\begin{theorem}[Pommerenke, \cite{Po}]\label{thm:Pom1}
Let $f$ be an analytic self-map of $\Hh$ of parabolic type with Denjoy-Wolff point infinity and $z_n=x_n+iy_n=f_n(1)$ be a forward iteration sequence. Then the normalized iterates 
\[
\psi_n(z)=\frac{f_n(z)-iy_n}{x_n}
\]
converge uniformly on compact subsets of $\Hh$ to a function $\psi$ such that $\psi(\Hh)\subseteq\Hh$, $\psi(1)=1$ and
\begin{align*}
\psi\circ f(z)=\phi\circ\psi(z),\hspace{0.5in} \forall z\in\Hh;
\end{align*}
where $\phi$ is a M\"obius transformation of $\Hh$ into itself and $\psi(\infty)=\infty$.
In particular, if $f$ is parabolic non-zero-step type, $\phi(z)=z+ib$, i.e. translation in $\Hh$, and if $f$ is parabolic zero-step, $\phi(z)\equiv z$ and $\psi(z)\equiv 1$.
\end{theorem}

Non-trivial conjugation for parabolic zero-step case was discovered at \cite{BPo}:

\begin{theorem}[Baker and Pommerenke]
Let $f$ be an analytic self-map of $\Hh$ of parabolic zero-step type with the Denjoy-Wolff point infinity. Then the sequence of normalized functions
\[
\psi_n(z)=\frac{f_n(z)-z_n}{z_{n+1}-z_n}
\]
converges uniformly on compact subsets of $\Hh$ to a function $\psi$ such that 
\begin{align*}
\psi(f(z))=\psi(z)+1,\hspace{0.5in} \forall z\in\Hh.
\end{align*}
\end{theorem}
 
Thus parabolic zero-step maps can be conjugated to a (horizontal) shift in the plane.

Geometrically, these two types of maps differ by how forward iterates approach the Denjoy-Wolff point. In the parabolic non-zero-step case they converge to the Denjoy-Wolff point tangentially (Remark 1, \cite{Po}), see Figure \ref{fig:parabolic_tang}. 
\begin{figure}[h]
	\centering
		\includegraphics{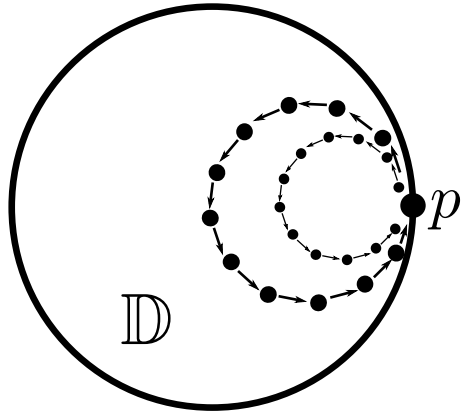}
	\caption{Orbits converge tangentially to the Denjoy-Wolff point $p$ in the parabolic non-zero-step case.}
	\label{fig:parabolic_tang}
\end{figure}
In the parabolic zero-step case, forward iterates may converge radially (Figure \ref{fig:parabolic_rad}), but a complete classification of their behavior has still not been achieved.
\begin{figure}[ht]
	\centering
		\includegraphics{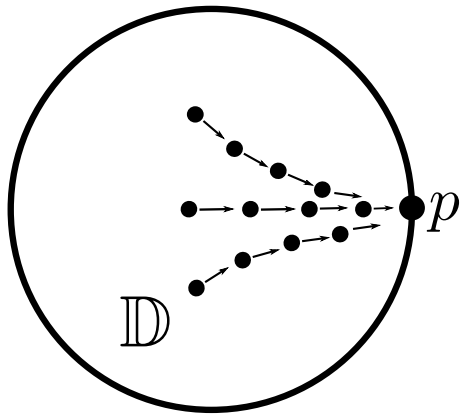}
	\caption{In some cases, orbits converge radially to the Denjoy-Wolff point $p$ in the parabolic zero-step case.}
	\label{fig:parabolic_rad}
\end{figure}

\section{Multi-dimensional case}

Now consider self map $f$ of $N$-dimensional unit ball $\B^N=\left\{Z\in\C^N: \|Z\|<1\right\}$.

Schwarz's lemma still holds in $\B^N$, with pseudo-hyperbolic distance defined as
\begin{align}
d_{\B^N}(Z,W):=\left(\frac{\left|1-\left\langle Z,W\right\rangle\right|^2}{\left(1-\left\|Z\right\|^2\right) \left(1-\left\|W\right\|^2\right)}\right)^{1/2}.\label{dBN}
\end{align}

And a version of the Denjoy-Wolff theorem also holds:
\begin{theorem}[Herv\'{e} \cite{Herve}, MacCluer \cite{McCl}]\label{thm:DWN}
Let $f:\B^N\to\B^N$ be a holomorphic map without fixed points in $\B^N$. Then the sequence of iterates $\{f_n\}$ converges uniformly on compact subsets of $\B^N$ to the constant map $Z\mapsto p$ for a (unique) point $p\in\partial\B^N$ (called the Denjoy-Wolff point of $f$); and the number
\begin{align*}
c:=\liminf_{Z\to p}\frac{1-\|f(Z)\|}{1-\|Z\|}\in(0,1]
\end{align*}
is called the multiplier or the boundary dilatation coefficient of $f$ at $p$.
\end{theorem}
The map $f$ is called {\bf hyperbolic} if $c<1$ and {\bf parabolic} if $c=1$.

For the maps of parabolic type, it is still possible to define zero-step and non-zero step sequences as in Definition \ref{def:znz}. However, the question whether the same map can have sequences of both types is still open.

\begin{conjecture}
Let $f$ a self map of $\B^N$ of parabolic type. If the step $d_{\B^N}(f_n(Z_0),f_{n+1}(Z_0))\to 0$ for some $Z_0\in\B^N$, then $d_{\B^N}(f_n(Z),f_{n+1}(Z))\to 0$ for all $Z\in\B^N$.
\end{conjecture}

To describe geometric behavior of forward iterates near the boundary of the ball, we will need several notions that generalize non-tangential approach in the disk:

\begin{definition}
The Koranyi region $K(X,M)$ of vertex $X\in\partial\B^N$ and amplitude $M>1$ is the set
\begin{align*}
K(X,M)=\left\{Z\in\B^N\left| \frac{\left|1-\left\langle Z,X\right\rangle\right|}{1-\|Z\|}<M\right.\right\}.
\end{align*}
\end{definition}

When $N=1$, it is the usual Stolz angle in the disk; but for $N>1$ the region is tangent to the boundary of the ball along some directions. 

\begin{definition}\label{def:restrictedcurve}
For $X\in\partial\B^N$, a sequence $Z_n\to X$ is called {\bf special} if 
\begin{align}
\lim_{n\to\infty}\frac{\|Z_n-\left\langle Z_n,X\right\rangle X\|^2}{1-\|\left\langle Z_n,X\right\rangle X\|^2}=0,\label{specialcond}
\end{align}
and {\bf restricted} if it is special and its orthogonal projection $\left\langle Z_n,X\right\rangle X$ is non-tangential.
\end{definition}

The connection between non-tangential, Koranyi and restricted approaches is described by the following
\begin{lemma}[Lemma (2.2.24), \cite{Abate}]\label{spec_tang_Kor}
Let $Z_n\in\B^N$ be a sequence such that $Z_n\to X\in\partial\B^N$ as $n\to\infty$. Then
\begin{enumerate}[(i)]
	\item{if $Z_n$ is non-tangential, then it is restricted;}
	\item{assume $Z_n$ is special. If $Z_n$ is restricted, then it lies eventually in a Koranyi region with vertex $X$. Conversely, if $Z_n$ lies in a Koranyi region, it is restricted.}
\end{enumerate}
\end{lemma}

Thus Koranyi and restricted regions are weaker generalizations of the non-tangential approach in one dimension. In this paper, we consider restricted sequences of forward iterates in the parabolic case.

Similarly to the non-tangential approach the one-dimensional case, Koranyi regions and restricted approaches are used to define limits at the boundary of the ball, called $K$-limits and restricted $K$-limits, respectively.

We will need the following result:
\begin{lemma}[part (i) of Theorem 2.2.29 in \cite{Abate}]\label{thm:ang_lder}
Let $f$ be an analytic self-map of $\B^N$ and $X\in\partial\B^N$ be such that
\[
\liminf_{Z\to X}\frac{1-\left\|f(Z)\right\|}{1-\left\|Z\right\|}=\alpha<\infty.
\]
Then $f$ has $K$-limit $Y\in\partial\B^N$, and the function
\begin{align*}
\frac{1-\left\langle f(Z),Y\right\rangle}{1-\left\langle Z,X\right\rangle}
\end{align*}
has restricted $K$-limit $\alpha$ at $X$ and is bounded in every Koranyi region.
\end{lemma}

We obtained the following result for the forward iteration sequences in the unit ball:

\begin{theorem}\label{claim:zerostep}
Let $f$ be a parabolic self-map of the unit ball $\B^N$ with the Denjoy-Wolff point $(1,0)\in\C\times\C^N$. If the sequence of forward iterates $\left\{Z_n\right\}_{n=1}^{\infty}$ is restricted, then it must have zero step, i.e., $d_{\B^N}(Z_n,Z_{n+1})\to 0$.
\end{theorem}
\begin{proof}
Denote $Z_n=(z_n,w_n)\in \C \times\C^{N-1}$. Since ${Z_n}$ is restricted, it is special, i.e.,
\begin{align}
\lim_{n\to\infty}\frac{\|w_n\|^2}{1-|z_n|^2}=0\label{special}
\end{align}
and projections on the first dimension ${z_n}$ tend to $1$ non-tangentially. Moreover, by Lemma \ref{thm:ang_lder},
\begin{align}
\lim_{n\to\infty}\frac{1-z_{n+1}}{1-z_n}=1.\label{radialder}
\end{align}
By (\ref{dBN}), the pseudo-hyperbolic distance satisfies

\begin{align*}
1-d^2_{\B^N}(Z_n,Z_{n+1})&=\frac{(1-\left\|Z_n\right\|^2)(1-\left\|Z_{n+1}\right\|^2)}{\left|1-\left\langle Z_n,Z_{n+1}\right\rangle\right|^2}\\ &=\frac{(1-|z_n|^2-\|w_n\|^2)(1-|z_{n+1}|^2-\|w_{n+1}\|^2)}{\left|1-z_n\overline{z_{n+1}}-\left\langle w_n,w_{n+1}\right\rangle\right|^2}\\
&=\frac{(1-\frac{\|w_n\|^2}{1-|z_n|^2})(1-\frac{\|w_{n+1}\|^2}{1-|z_{n+1}|^2})}{\left|\frac{1-z_n\overline{z_{n+1}}}{\sqrt{1-|z_n|^2}\sqrt{1-|z_{n+1}|^2}}-\left\langle \frac{w_n}{\sqrt{1-|z_n|^2}},\frac{w_{n+1}}{\sqrt{1-|z_{n+1}|^2}}\right\rangle\right|^2}
\end{align*}

By (\ref{special}), it is enough to show that 

\begin{align*}
\left|\frac{1-z_n\overline{z_{n+1}}}{\sqrt{1-|z_n|^2}\sqrt{1-|z_{n+1}|^2}}\right|\to 1
\end{align*}

which is equivalent to $d(z_n,z_{n+1})\to 0$.

By definition,
\begin{align*}
d(z_n,z_{n+1})=\left|\frac{z_{n+1}-z_n}{1-\overline{z_n}z_{n+1}}\right|=\left|\frac{1-z_n-1+z_{n+1}}{1-\overline{z_n}+\overline{z_n}-\overline{z_n}z_{n+1}}\right|=\left|\frac{1-\frac{1-z_{n+1}}{1-z_n}}{\frac{1-\overline{z_n}}{1-z_n}+\overline{z_n}\frac{1-z_{n+1}}{1-z_n}}\right|
\end{align*}

By (\ref{radialder}), it is enough to show that the denominator is bounded away from $0$, which is indeed the case when $\frac{1-\overline{z_n}}{1-z_n}$ is bounded away from $-1$.

But we have
\begin{align*}
\Arg\left(\frac{1-\overline{z_n}}{1-z_n}\right)=-2\Arg(1-z_n)\geq -\pi+\epsilon,
\end{align*}
for some $\epsilon>0$, because $z_n\to 1$ non-tangentially, and thus $\frac{1-\overline{z_n}}{1-z_n}$ stays away from $-1$.
\end{proof}

\begin{remark}
Since any non-tangential approach must be restricted (Lemma \ref{spec_tang_Kor}), it follows that every non-zero-step sequence must converge tangentially, and Theorem \ref{claim:zerostep} is a generalization of the classical one-dimensional result (Remark 1, \cite{Po}). 
\end{remark}

\bibliographystyle{amsplain}
\bibliography{parabolic_ref}
\end{document}